\documentclass[12pt,reqno]{amsart}

\usepackage[T1]{fontenc}
\usepackage[latin1]{inputenc}
\usepackage[english]{babel}

\usepackage{tikz-cd}
\tikzset{%
    symbol/.style={%
        draw=none,
        every to/.append style={%
            edge node={node [sloped, allow upside down, auto=false]{$#1$}}}
    }
}

\usepackage{amsmath,amssymb}
\usepackage{amsthm}
\usepackage{thmtools}

\newtheorem{lemma}{Lemma}
\newtheorem{thm}{Theorem}
\newtheorem{prop}{Proposition}
\newtheorem{coro}{Corollary}
\newtheorem*{conv}{Convention}
\theoremstyle{definition}
\newtheorem{defn}{Definition}

\newtheorem{example}{Example}
\newtheorem{remark}{Remark}

\newcommand{\C}{\mathcal{C}}

\newcommand{\E}{\mathcal{E}}
\newcommand{\F}{\mathcal{F}}
\newcommand{\U}{\mathbb{U}}
\newcommand{\V}{\mathbb{V}}
\newcommand{\W}{\mathbb{W}}
\newcommand{\X}{\mathbb{X}}
\newcommand{\Y}{\mathbb{Y}}
\newcommand{\Z}{\mathbb{Z}}
\newcommand{\Xc}{\mathcal{X}}

\newcommand{\RG}{\mathbf{RG}}
\newcommand{\Grpd}{\mathbf{Grpd}}

\newcommand{\Span}{\mathbf{Span}}
\newcommand{\Conn}{\mathbf{Conn}}
\newcommand{\Grp}{\mathbf{Grp}}
\newcommand{\Ab}{\mathbf{Ab}}

\newcommand{\Eq}{\mathbf{Eq}}
\newcommand{\Comp}{\mathbf{Comp}}
\newcommand{\Ext}{\mathbf{Ext}}
\newcommand{\CExt}{\mathbf{CExt}}
\newcommand{\Lie}{\mathbf{Lie}}
\newcommand{\XMod}{\mathbf{XMod}}
\newcommand{\PXMod}{\mathbf{PXMod}}

\usepackage{enumitem}

\newlist{thmlist}{enumerate}{1}
\setlist[thmlist]{label=(\arabic{thmlisti}), ref=\thethm.(\arabic{thmlisti}),noitemsep}

\usepackage[capitalize]{cleveref}

\Crefname{thm}{Theorem}{Theorems}
\Crefname{lemma}{Lemma}{Lemmas}
\Crefname{prop}{Proposition}{Propositions}
\Crefname{coro}{Corollary}{Corollaries}
\Crefname{defn}{Definition}{Definitions}
\Crefname{counter}{Counterexample}{Counterexamples}
\Crefname{example}{Example}{Examples}

\addtotheorempostheadhook[thm]{\crefalias{thmlisti}{thm}}
\addtotheorempostheadhook[prop]{\crefalias{thmlisti}{prop}}
\addtotheorempostheadhook[coro]{\crefalias{thmlisti}{coro}}

\begin{document}
\title{Higher commutator conditions for extensions in Mal'tsev categories}
\author{Arnaud Duvieusart}
\thanks{The first author is a Research Fellow of the Fonds de la Recherche Scientifique-FNRS}
\author{Marino Gran}
\address{Institut de Recherche en Math\'ematique et Physique, Universit\'e Catholique de Louvain, Chemin du Cyclotron 2,
1348 Louvain-la-Neuve, Belgium}

\email{arnaud.duvieusart@uclouvain.be}
\email{marino.gran@uclouvain.be}
\begin {abstract}
We define a Galois structure on the category of pairs of equivalence relations in an exact Mal'tsev category, and characterize central and double central extensions in terms of higher commutator conditions. These results generalize both the ones related to the abelianization functor in exact Mal'tsev categories, and the ones corresponding to the reflection from the category of internal reflexive graphs to the subcategory of internal groupoids. Some examples and applications are given in the categories of groups, precrossed modules, precrossed Lie algebras, and compact groups.
\end {abstract}

%\infonum{?}{15}

\date{\today}

\maketitle

\section*{Introduction}
During the last thirty years categorical Galois theory has provided several new insights into the study of (higher) central extensions of groups, Lie algebras, precrossed modules and general algebraic varieties of universal algebras \cite{J91, JK94, JK97, EG06, EGV, E14}. An important idea in this approach, that is also related to the work of Fr\"{o}hlich \cite{Fro} in the special case of varieties of $\Omega$-groups (i.e. groups with multiple operators in the sense of Higgins), is to define a notion of \emph{centrality} in a variety of algebras $\C$ that is \emph{relative} to the choice of a subvariety $\Xc$ of $\C$, with reflector $I \colon \C \rightarrow \Xc$, left adjoint to the inclusion functor  $\Xc \rightarrow \C$. 
One calls a homomorphism $f \colon X \rightarrow Y$ in the variety $\C$ an \emph{extension} when it is \emph{surjective}; if $\eta \colon 1_{\mathcal C} \rightarrow I$ denotes the unit of the adjunction, then such an extension $f \colon X \rightarrow Y$ is \emph{$\Xc$-trivial} when the naturality square 
\begin{equation}\begin{tikzcd}[ampersand replacement=\&,column sep=huge] X\ar[d,"f"']\ar[r,"\eta_X"] \& IX \ar[d,"I(f)"] \\ Y \ar[r,"\eta_Y"'] \& IY \end{tikzcd}\end{equation}
is a pullback. It is called an $\Xc$-\emph{central extension} when there exists a surjective homomorphism $p \colon E \to Y$ such that the homomorphism $\pi_1\colon E \times_Y X \to E$ in the pullback 
\begin{equation}\begin{tikzcd}[ampersand replacement=\&,column sep=huge] E \times_Y X\ar[d,"\pi_1"']\ar[r,"\pi_2"] \& X \ar[d,"f"] \\ E \ar[r,"p"'] \& Y \end{tikzcd}\end{equation}
 is a trivial extension. It turns out that, in the special case of the reflection $ab \colon \Grp \to \Ab$ from the variety $\Grp$ of groups to its subvariety $\Ab$ of abelian groups, a surjective group homomorphism $f \colon X \rightarrow Y$ is $\Ab$-\emph{central} exactly when its kernel is contained in the center of $X$ or, equivalently, when the commutator condition $[Ker(f), X] = \{ 1\}$ holds.
Similar results hold in other varieties of algebras, such as the variety $B$-$\mathbf{PXMod}$
 of $B$-\emph{precrossed modules}, i.e. precrossed modules with the same codomain $B$, where morphisms $f \colon (X, \alpha) \to (Y, \alpha')$ are action preserving group homomorphisms making the triangle  
 \[\begin{tikzcd}[ampersand replacement=\&] X \ar[rr,"f"] \ar[dr,"{\alpha}" ']  \& \& Y \ar[dl,"\alpha'"] \\  \& B \&\end{tikzcd}\]
 commute \cite{CE89}.
 As shown in \cite{EG06} an extension $f \colon (X, \alpha) \to (Y, \alpha')$ of $B$-precrossed modules 
 in the variety $B$-$\mathbf{PXmod}$ is $\Xc$-central relatively to the subvariety $\Xc$ = $B$-$\mathbf{Xmod}$ of $B$-crossed modules if and only if the \emph{Peiffer commutator} condition $\langle Ker(f), X \rangle = \{1\}$ holds. A much more general result involving the Smith commutator of congruences was established in \cite{EG06} in the context of Mal'tsev varieties \cite{S76}, namely those varieties of universal algebras whose algebraic theory contains a ternary term $p(x,y,z)$ such that $p(x,y,y)=x$ and $p(x,x,y) = y$. For instance, the term $p(x,y,z) = x \cdot y^{-1} \cdot z$ in the algebraic theory of groups shows that the variety $\Grp$ of groups is a Mal'tsev variety.
The central extensions of reflexive graphs in a Mal'tsev variety $\C$ 
\begin{equation}\begin{tikzcd}[ampersand replacement=\&] X_1 \ar[rr,"f"] \ar[dr,shift left,"d"] \ar[dr,shift right,"c"'] \& \& Y_1 \ar[dl,shift left,"d'"] \ar[dl,shift right,"c'"'] \\  \& B \&\end{tikzcd}\label{extensiongraph}\end{equation}
 in the category $\RG(\C)/B$ of reflexive graphs in a Mal'tsev variety $\C$ (over a fixed $B$ in $\C$) were shown to be central with respect to the reflection to the subcategory $\Grpd(\C)/B$ of groupoids in $\C$ (over $B$) exactly when the \emph{commutator condition}
\[ [Eq[f],Eq[c]\vee Eq[d]]=\Delta_{X_1} \]
holds. Here $c$ and $d$ are the ``codomain'' and the ``domain'' homomorphisms, $Eq[c]\vee Eq[d]$ is the supremum of the congruences $Eq[c]$ and $Eq[d]$ (occuring as the kernel pairs of $c$ and $d$, respectively) and $\Delta_{X_1}$ is the smallest congruence on the algebra $X_1$.
By using the known equivalence between the categories of internal reflexive graphs and of internal groupoids in the variety of groups with the categories of precrossed modules and crossed modules, respectively \cite{BS, LR80} one easily gets the result concerning $B$-precrossed modules recalled above as a very special case of this one.

In this article we always work in a general exact Mal'tsev category \cite{CLP91} (with coequalizers), so that all our results hold not only in any Mal'tsev variety of algebras, but also in many other categories, such as the categories of compact Hausdorff Mal'tsev algebras, of ${\mathbb C}^*$-algebras \cite{GRos}, and the dual category of any elementary topos \cite{CKP93}, for instance. In this context there is a good theory of commutators of equivalence relations \cite{P95} that we briefly recall in Section \ref{preliminaries}.
We then provide a conceptual proof of a general result that yields, in particular, the characterization of central extensions of reflexive graphs recalled above in the more general context of exact Mal'tsev categories (see Theorem \ref{Main2}). The case when $\C$ is the category of compact Hausdorff groups is explained in detail in Example \ref{compact}. We then study the so-called \emph{double central extensions} in the general context of exact Mal'tsev categories, and particularize our results to deduce a precise characterization of the double central extensions relative to the reflection $\RG(\C)/B \to \Grpd(\C)/B$  in terms of a commutator condition involving the categorical commutator of equivalence relations (see Corollary \ref{doublecentralgraph}). Our general result also includes the characterization of double central extensions relative to the abelianization functor previously considered by T. Everaert and T. Van der Linden in \cite{EV10} (see Corollary \ref{double-abelian}). Some applications in the special case of central extensions of precrossed Lie algebras are also given (\Cref{Lie}).

In a recent article A. Cigoli, S. Mantovani and G. Metere have investigated a categorical notion of Peiffer commutator that is certainly related to the results presented in the present work, in the special case of semi-abelian categories. In the future it would be interesting to compare the categorical commutator conditions arising from the reflection $\RG(\C)/B \to \Grpd(\C)/B$ considered in our work with this new notion of Peiffer commutator \cite{CMM}.

\section{Preliminaries}\label{preliminaries}
A \emph{Mal'tsev category} \cite{CLP91} is a finitely complete category $\C$ having the property that any reflexive relation in $\C$ is an equivalence relation. In this paper we shall always assume that $\C$ is a regular Mal'tsev category, so that it will be possible to consider the direct image $f(R)$ of any equivalence relation $R$ on $X$ along an arrow $f : X \to X'$.

If $\Eq(\C)$ denotes the category of (internal) equivalence relations in $\C$, we write $2$-$\Eq(\C)$ for the category whose objects are triples $(X,R,S)$ where $R$ and $S$ are two equivalence relations on an object $X$, and arrows $(X,R,S)\to (X',R',S')$ are those arrows $f \colon X\to X'$ in $\C$ such that $f(R)\leq R'$ and $f(S)\leq S'$. This means that there exist arrows $f_R:R\to R'$ and $f_S:S\to S'$ in $\C$ such that the following diagram commutes:

\begin{equation}\begin{tikzcd}[ampersand replacement=\&] R \ar[r,shift left]\ar[r,shift right] \ar[d,"f_R"'] \& X \ar[d,"f"] \& \ar[l,shift left]\ar[l,shift right] S \ar[d,"f_S"]\\
R' \ar[r,shift left]\ar[r,shift right] \& X' \& \ar[l,shift left]\ar[l,shift right] S' \end{tikzcd}\label{eq:2_equiv}\end{equation}

One can check that limits in $2$-$\Eq(\C)$ are just pointwise limits in $\C$, and the same holds for regular epimorphisms; in other words, regular epimorphisms in $2$-$\Eq(\C)$ are regular epimorphisms $f \colon X\to X'$ such that $f(R)=R'$ and $f(S)=S'$. The categories 
$\Eq(\C)$ and $2$-$\Eq(\C)$ are easily seen to be regular Mal'tsev categories whenever $\C$ is a regular Mal'tsev category \cite{BG03}.

As in \cite{JK94}, we denote $\C \downarrow X$ the full subcategory of the usual comma category $\C/X$ whose objects are the regular epimorphisms with codomain $X$. Then, since regular epimorphisms are pullback stable in the regular category $\mathcal C$, any $f:X\to Y$ induces a change-of-base functor $f^*:\C\downarrow Y\to \C\downarrow X$ defined by pulling back along $f$. When $f$ is a regular epi, this functor $f^*$ has a left adjoint $f_!$ defined by composition with $f$. When $\C$ is (Barr) exact, any regular epimorphism $f$ is an effective descent morphism, in the sense that $f^*$ is monadic; the same holds for the category $\Eq(\C)$ of equivalence relations in $\C$ (\cite{JS11}). Because the two equivalence relations in our definition of $2$-$\Eq(\C)$ do not interact when forming limits or coequalizers, the same is true in $2$-$\Eq(\C)$.

Given two equivalence relations $(R,r_1,r_2)$ and $(S,s_1,s_2)$ on the same object $X$ in $\C$, a \emph{double equivalence relation} on $R$ and $S$ is an equivalence relation in $\Eq(\C)$, depicted as
\begin{equation}\begin{tikzcd}[ampersand replacement=\&] C \ar[r,shift left,"q_1"]\ar[r,shift right,"q_2"'] \ar[d,shift left,"p_2"]\ar[d,shift right,"p_1"'] \& S \ar[d,shift left,"s_2"]\ar[d,shift right,"s_1"'] \\ R \ar[r,shift left,"r_1"]\ar[r,shift right,"r_2"'] \& X, \end{tikzcd}\label{eq:double_relation}\end{equation}
where $(C,p_1,p_2)$ is an equivalence relation on $R$, $(C,q_1,q_2)$ is an equivalence relation on $S$, and $r_ip_j=s_jq_i$ for all $i,j\in \{1,2\}$. A double equivalence relation on $R$ and $S$ is said to be \emph{centralizing} if any of the commutative squares in diagram \eqref{eq:double_relation} is a pullback (and in this case all of them will be). Two equivalence relations $R$ and $S$ are said to \emph{centralize each other} (or to be \emph{connected}) if they have a centralizing relation. We shall write $R\square S$ for the largest double equivalence relation on $R$ and $S$ (this latter can be constructed as a limit of a suitable diagram \cite{JP01}).

In a finitely complete category it is often useful to denote ``generalized elements'' of a double relation as $\left(\begin{smallmatrix} a & b \\ d & c\end{smallmatrix}\right)$, where $a,b,c,d$ in $X$ are such that $aRb$, $dRc$, $aSd$, $bSc$. With this notation, a double relation $C$ is centralizing if for all $a,b,c$ such that $aRb$ and $bSc$, there exists a unique $d$ such that $\left(\begin{smallmatrix} a & b \\ d & c\end{smallmatrix}\right)\in C$.

We shall consider the pullback
\[\begin{tikzcd}[ampersand replacement=\&] R\times_X S \ar[r,"\pi_2"]\ar[d,"\pi_1"]\& S \ar[d,"s_1"]\\ R\ar[r,"r_2"'] \& X. \end{tikzcd}\]
When $\C$ is a regular Mal'tsev category, then the following conditions are equivalent (see \cite{BG03,CPP92,P95,P96}):
\begin{itemize}
\item $R$ and $S$ have a unique centralizing relation;
\item there exists a unique arrow $\beta:R\times_X S\to R\square S$ such that the diagram
\[\begin{tikzcd}[ampersand replacement=\&] R\square S \ar[drr,bend left,shift left,"q_2"]\ar[ddr,bend right,shift left,"p_1"] \ar[dr,shift left,"\alpha"] \& \& \\
\& R\times_X S \ar[r,"\pi_2",shift left]\ar[d,"\pi_1",shift left] \ar[ul,shift left,"\beta"]\& S \ar[l,shift left,"\sigma_2"] \ar[d,shift left,"s_1"]\ar[ull,bend right,shift left,"\delta_Q"] \\
\& R \ar[r,shift left,"r_2"] \ar[uul,bend left, shift left,"\delta_P"] \ar[u,shift left,"\sigma_1"] \& X \ar[l,shift left,"\delta_R"] \ar[u,shift left,"\delta_S"]\end{tikzcd}\]
commutes, where
\[\alpha \beta= id_{R\times_X S}, \quad \pi_1\sigma_1=id_R, \quad \pi_2\sigma_2=id_S,\]
and the arrow $\delta_R$ (resp. $\delta_S$, $\delta_P$, $\delta_Q$) is the diagonal of the relation $R$ (resp. $S$, $P=(R\square S,p_1,p_2)$, $Q=(R\square S ,q_1,q_2)$), that exists since $R\square S$ is a double equivalence relation. 
\item there exists a unique partial Mal'tsev operation $p:R\times_X S \to X$, called a \emph{connector} in \cite{BG02,BG03}, satisfying 
\[p(x,y,y)=x \text{ and } p(y,y,z)=z\]
\end{itemize}
In addition, the connector $p$ has the property that $xSp(x,y,z)$ and $zRp(x,y,z)$, and also satisfies the "associativity axiom" $p(p(x,y,z),u,v)=p(x,y,p(z,u,v))$ whenever it makes sense. Thus if $\left(\begin{smallmatrix} a & b \\ d & c\end{smallmatrix}\right)$ is an element of $R\square S$, then $(p(a,b,c),d)\in R\wedge S$, and the following square is a pullback 

\[\begin{tikzcd}[ampersand replacement=\&] R \square S \ar[d,"q"'] \ar[r,"\alpha"] \& R\times_{X} S \ar[d,"p"] \\
R\wedge S \ar[r,"p_1"'] \& X,
\end{tikzcd}\]
where $\alpha \colon R \square S \rightarrow R\times_{X} S$ associates $(a,b,c)$ with any $\left(\begin{smallmatrix} a & b \\ d & c\end{smallmatrix}\right)$ in $R\square S$.
Let us denote $\Conn(\C)$ the full subcategory of $2$-$\Eq(\C)$ consisting of triples $(X,R,S)$ where $R$ and $S$ are equivalence relations centralizing each other. It is proved in \cite{BG03} that when $\C$ is a regular Mal'tsev category, the subcategory $\Conn(\C)$ is closed in $2$-$\Eq(\C)$ under regular quotients and subobjects in $\C$.

In an exact Mal'tsev category $\C$ with coequalizers, M.C. Pedicchio introduced in \cite{P95} a definition of commutator of two equivalence relations $R$ and $S$ on the same object $X$, which we denote $[R,S]$, extending the Smith commutator of congruences (which is defined for varieties of universal algebras \cite{S76}). This categorical commutator can be characterized as the smallest equivalence relation on $X$ whose coequalizer $q$ is such that $q(R)$ and $q(S)$ centralize each other. In particular, it has the following properties for all equivalence relations $R,S,T$ on the same object $X$ (see \cite{P95,P96}):
\begin{enumerate}
\item $[R,S]=[S,R]$;
\item $[R,S]\leq R\wedge S$;
\item if $S\leq T$, then $[R,S]\leq [R,T]$;
\item $[R,S\vee T]=[R,S]\vee [R,T]$;
\item for any regular epimorphism $f$, $[f(R),f(S)]=f([R,S])$.
\end{enumerate}

This characterization of the commutator implies that $\Conn(\C)$ is a reflective subcategory of the regular Mal'tsev category $2$-$\Eq(\C)$, and the $(X,R,S)$-component of the unit of the corresponding reflection is induced by the canonical quotient $X \to \frac{X}{[R,S]}$ of $X$ by the equivalence relation $[R,S]$. $\Conn(\C)$ is thus a \emph{Birkhoff subcategory} (in the sense of \cite{JK94}) of $2$-$\Eq(\C)$. In general, for a full reflective subcategory $\mathcal{X}$ of a regular category $\C$, with reflector $I \colon \C \rightarrow \mathcal{X}$, the property of being stable under subobjects and quotients is equivalent to the fact that for any regular epimorphism $f:X\to Y$ of $\C$ the naturality square
\begin{equation}\label{eq:square}\begin{tikzcd}[ampersand replacement=\&,column sep=huge] X\ar[d,"f"']\ar[r,"\eta_X"] \& IX \ar[d,"I(f)"] \\ Y \ar[r,"\eta_Y"'] \& IY \end{tikzcd}\end{equation}
induced by the units of the adjunction is a pushout of regular epimorphisms in $\C$ (see Proposition $3.1$ in \cite{JK94}). When $\C$ is an exact Mal'tsev category, this is equivalent to the fact that \eqref{eq:square} is a \emph{regular pushout} (by Theorem $5.7$ in \cite{CKP93}): this means that all the arrows in the diagram 
\[\begin{tikzcd}[ampersand replacement=\&] X \ar[drr,bend left,shift left,"\eta_X"] 
\ar[ddr,"f"',bend right] \ar[dr,shift left,"\alpha"] \& \& \\
\& Y \times_{IY} IX  \ar[r,"p_2",shift left]\ar[d,"p_1",shift left] \& IX  \ar[d,shift left,"I(f)"] \\
\& Y\ar[r,"\eta_Y"']  \& IY, \end{tikzcd}\]
where $(Y \times_{IY} IX, p_1, p_2)$ is the pullback of $\eta_Y$ and $I(f)$, are regular epimorphisms. In general, however, the square \eqref{eq:square} being a regular pushout is a stronger property than the one of being a pushout of regular epimorphisms. If $\mathcal E$ is the class of regular epimorphisms such a full reflective subcategory is called a \emph{strongly $\mathcal E$-Birkhoff subcategory} \cite{EGV}.

\begin{conv} From now on $\C$ will always denote an exact Mal'tsev category with coequalizers.\end{conv}

\begin{prop} If $f:(X,R,S)\to (X',R',S')$ is a regular epimorphism in $2$-$\Eq(\C)$, then the square
\[\begin{tikzcd}[ampersand replacement = \& ] X\ar[r,"q_{[R,S]}"]\ar[d,"f"'] \& \frac{X}{[R,S]} \ar[d,"\bar{f}"] \\ X' \ar[r,"q_{[R,S]}"'] \& \frac{X'}{[R',S']}
\end{tikzcd}\]
induces a regular pushout in $2$-$\Eq(\C)$.
The category $\Conn(\C)$ is then a strongly $\mathcal E$-Birkhoff subcategory of $2$-$\Eq(\C)$.
\label{prop:Birkhoff}\end{prop}

\begin{proof} Since pullbacks and regular epimorphisms in $2$-$\Eq(\C)$ are degreewise pullbacks and regular epimorphisms in $\C$, so are regular pushouts. Let us consider the following commutative diagram:
\[\begin{tikzcd}[ampersand replacement=\&] \& \& R \ar[dd,shift left]\ar[dd, shift right] \ar[rr] \ar[dr,"f_R"]\& \& q_{[R,S]}(R) \ar[dd,shift left]\ar[dd, shift right] \ar[dr] \& \\
\& \& \& {R'} \ar[rr,crossing over] \& \& q_{[R',S']}(R') \ar[dd,shift left]\ar[dd, shift right] \\
{[R,S]} \ar[dr,"\tilde{f}"']\ar[rr,shift left]\ar[rr, shift right] \& \& X\ar[rr,"q_{[R,S]}", near end] \ar[dr,"f"'] \ar[dd,"q_R"' near end] \& \& \frac{X}{[R,S]} \ar[dr,"\bar{f}"] \ar[dd]\&  \\
\& {[R',S']} \ar[rr,shift left, crossing over]\ar[rr, shift right,crossing over] \& \& {X'}\ar[rr,"q_{[R',S']}" near start, crossing over] \ar[from=uu,shift left,crossing over]\ar[from=uu, shift right,crossing over]\& \& \frac{X'}{[R',S']} \ar[dd] \\
\& \& \frac{X}{R} \ar[dr] \ar[rr,equal]\& \& \frac{X}{R} \ar[dr]\& \\
\& \& \& \frac{X'}{R'} \ar[from=uu, crossing over,"q_{R'}" near start] \ar[rr,equal] \& \& \frac{X'}{R'}
\end{tikzcd}\]
By the properties of the commutator one has the equalities
\[[R',S']=[f(R),f(S)]=f([R,S]),\]
thus the arrow $\tilde{f}:[R,S]\to [R',S']$ is a regular epimorphism. It follows that the square
\[\begin{tikzcd}[ampersand replacement=\&] X \ar[d,"f"'] \ar[r,"q_{[R,S]}"]\& \frac{X}{[R,S]} \ar[d,"\bar{f}"]\\ X' \ar[r,"q_{[R',S']}"']\& \frac{X'}{[R',S']}\end{tikzcd}\]
is a regular pushout in $\C$. All the sides of the top face are regular epimorphisms; it suffices then to show that the comparison arrow to the pullback is a regular epi. To this end, notice that the bottom face is a pullback, since two opposite sides are isomorphisms. Now by commutativity of limits with limits, the kernel pair of the induced arrow $X'\times_{X'/[R',S']} X/[R,S]\to X/R$ is the pullback $R'\times_{q_{[R',S']}(R')} q_{[R,S]}(R)$. It follows that the comparison arrow is a regular epimorphism, because the induced square of coequalizers
\[\begin{tikzcd}[ampersand replacement=\&] X \ar[r]\ar[d,"q_R"'] \& X'\times_{X'/[R',S']} X/[R,S]\ar[d] \\ \frac{X}{R} \ar[r,equal] \& \frac{X}{R}\end{tikzcd}\]
is a regular pushout.
\end{proof}

Let us consider a class $\E$ of arrows in $2$-$\Eq(\C)$ satisfying the following axioms :
\begin{enumerate}[label=(E\arabic*)]
\item any isomorphism is in $\E$;
\item $\E$ is stable under pullbacks;
\item $\E$ is closed under composition.
\end{enumerate}

In the following we will call \emph{extensions} the arrows in $\E$. 
\begin{remark}
Of course, a natural example of a class of extensions in the regular category $2$-$\Eq(\C)$ - i.e. satisfying the axioms (E1), (E2), and (E3) above - is provided by the class $\E$ of regular epimorphisms. From now on we shall adopt an axiomatic approach to the class $\E$ that will have several advantages, including the possibility of comparing the results in our paper to the ones obtained by Everaert in \cite{E14}.
\end{remark}
Assume then that $H \colon \Conn(\C) \to 2$-$\Eq(\C)$ denotes the forgetful functor, and the class $\E$ of extensions in $2$-$\Eq(\C)$ also satisfies the condition 
\begin{enumerate}[label=(G\arabic*)]
\item $HI(\E) \subset \E$
\end{enumerate}
(where $I:2\text{-}\Eq(\C)\to \Conn(\C)$ denotes the reflector). One then gets a Galois structure $\Gamma_\E=(2$-$\Eq(\C),\Conn(\C),I,\E)$ in the sense of \cite{JK97,E14}.

We recall that an extension $f:(X,R,S) \to (X',R',S')$ is said to be \emph{monadic} when the pullback functor
\[f^* : 2\text{-}\Eq(\C)\downarrow_\E (X',R',S') \to 2\text{-}\Eq(\C)\downarrow_\E (X,R,S)\]
is monadic, where $2$-$\Eq(\C)\downarrow_\E (X,R,S)$ denotes the full subcategory of the slice category $2$-$\Eq(\C)/ (X,R,S)$ whose objects are the extensions.

The hypotheses on the class $\E$ implies that $I$ induces for all $(X,R,S)$ in $2$-$\Eq(\C)$ a functor
\[I^{(X,R,S)}:2\text{-}\Eq(\C)\downarrow_\E (X,R,S)\to \Conn(\C)\downarrow_{\E}I(X,R,S);\]
moreover one can show that this functor has a right-adjoint $H^{(X,R,S)}$ defined by taking the pullback of any extension along the $(X,R,S)$-component of the unit. An object $(X,R,S)$ is said to be \emph{admissible} when this right-adjoint is fully faithful.

We say that an extension is \emph{$\Gamma_\E$-trivial} when the naturality square
\begin{equation}\label{eq:units_square}\begin{tikzcd}[ampersand replacement=\&,column sep=huge] X \ar[d,"f"']\ar[r,"q_{[R,S]}"] \& \frac{X}{[R,S]} \ar[d,"\bar{f}"] \\ X' \ar[r,"q_{[R',S']}"'] \& \frac{X'}{[R',S']} \end{tikzcd}\end{equation}
induced by the units of the adjunction is a pullback. By definition a \emph{$\Gamma_\E$-central extension} $f \colon (X,R,S) \rightarrow (X',R',S')$ is one for which there exists a monadic extension $g \colon (X'',R'',S'') \rightarrow (X',R',S')$ such that the pullback of $f$ along $g$ is $\Gamma_\E$-trivial; one also says that $f$ is split by $g$. Finally, an extension is said to be \emph{$\Gamma_\E$-normal} if it is monadic and it is split by itself; equivalently, one requires that the projections of the kernel pair of $f$ are trivial extensions.

The class $\E$ of extensions can also be seen as the class of objects of a category $\Ext_\E(2$-$\Eq(\C))$, where an arrow between extensions $f\to h$ is defined as a commutative square in $2$-$\Eq(\C)$
\begin{equation}\label{eq:double_ext}\begin{tikzcd}[ampersand replacement=\&] (X,R_1,S_1)\ar[d,"f"'] \ar[r,"g"]\& (Z,R_3,S_3)\ar[d,"h"] \\ (Y,R_2,S_2)\ar[r,"j"'] \& (W,R_4,S_4).\end{tikzcd}\end{equation}
Central extensions form a full subcategory of $\Ext_\E(2$-$\Eq(\C))$, that we shall denote by $\CExt_\E(2$-$\Eq(\C))$. 

We will also study extensions in $\Ext_\E(2$-$\Eq(\C))$, the \emph{double extensions} in the sense \cite{J91,GR04}. These are defined as commutative squares \eqref{eq:double_ext} such that all arrows in the diagram
\[\begin{tikzcd}[ampersand replacement=\&] X\ar[ddr,"f"',bend right]\ar[dr,"{\langle f,g\rangle}"] \ar[drr,"g",bend left] \& \& \\ \& Y\times_W Z \ar[r,"p_2"']\ar[d,"p_1"] \& Z\ar[d,"h"] \\ \& Y\ar[r,"j"'] \& W.\end{tikzcd}\]
are in $\E$. We will denote by $\E^1$ the class of double extensions (considered as a class of arrows in $\Ext_\E(2$-$\Eq(\C))$); it can be proven (see \cite{EGoV12}) that for any class $\E$ satisfying axioms (E1) to (E3), the class $\E^1$ as defined above satisfies the same properties in $\Ext_\E(2$-$\Eq(\C))$. Note that when $\E$ is the class of regular epimorphisms in $\C$, then a double extension in $\C$ is the same thing as a regular pushout.

We now consider two additional axioms for the class $\E$ of extensions:
\begin{enumerate}[label=(E\arabic*)]
\setcounter{enumi}{3}
\item if $f\circ g$ is in $\E$, then $f$ is in $\E$;
\item any square
\[\begin{tikzcd}[ampersand replacement=\&] X\ar[d,"f"'] \ar[r,"g",shift left]\& Z\ar[l,"s", shift left]\ar[d,"h"] \\ Y\ar[r,"j",shift left] \& W\ar[l,"t", shift left].\end{tikzcd}\]
where $f$ and $h$ are in $\E$ and $gs=id_Z$, $jt=id_W$, $hg=jf$ and $fs=th$, the induced arrow $\langle f,g\rangle : X\to Y\times_W Z$ is in $\E$.
\end{enumerate}

Then one can prove (see \cite{EGoV12}) that a square \eqref{eq:double_ext} of extensions is a double extension if and only if in the diagram
\begin{equation}\label{property}
\begin{tikzcd}[ampersand replacement=\&] Eq[g] \ar[r,shift left,"\pi_1"]\ar[r,shift right,"\pi_2"'] \ar[d,"\bar{f}"']\& X \ar[r,"g"] \ar[d,"f"'] \& Z \ar[d,shift right,"h"] \\ Eq[j] \ar[r,shift left,"\pi_1'"]\ar[r,shift right,"\pi_2'"'] \& Y \ar[r,"j"'] \& W,\end{tikzcd}
\end{equation}
where $Eq[g]$ (resp. $Eq[j]$) denotes the kernel pair of $g$ (resp. $j$), the induced arrow $\bar{f}$ is an extension. When $\E$ is the class of regular epimorphisms, this is equivalent to $f(Eq[g])=Eq[j]$.

In \cite{E14} T. Everaert proved that for any category $\C$, if $\E$ is a class of extensions satisfying the properties (E1) to (E5) and such that every extension is monadic, and $\mathcal{X}$ is a strongly $\E$-Birkhoff subcategory, then the Galois structure $(\C,\mathcal{X},I,\E)$ is admissible in the sense of categorical Galois theory (i.e. every object is admissible), every central extension is also normal, $\CExt_\E(\C)$ is a strongly $\E^1$-Birkhoff subcategory in $\Ext_\E(\C)$, and every extension in $\E^1$ is monadic. In particular, if we write $I^1$ for the reflector $\Ext_\E(\C)\to \CExt_\E(\C)$, then $(\Ext_\E(\C),\CExt_\E(\C),I^1,\E^1)$ is also an admissible Galois structure, and we can define \emph{double central extensions} as double extensions that are central for this induced Galois structure; we denote $\CExt_\E^2(\C)=\CExt_{\E^1}(\CExt_\E(\C))$ the category of double central extensions. By applying the result of Everaert once again, we obtain that $\CExt_\E^2(\C)$ is itself a strongly $\E^2$-Birkhoff subcategory of $\Ext_\E^2(\C)$. This means in particular that given a cube
\[\begin{tikzcd}[ampersand replacement=\&]
X\ar[dr,"\alpha"']\ar[dd,"f"']\ar[rr,"g"] \& \& Z \ar[dr,"\gamma"]\ar[dd,"h", near end] \& \\
\& X' \ar[rr,"g'",crossing over,near start]\& \& Z'\ar[dd,"h'"]\\
Y \ar[rr,"j",near end] \ar[dr,"\beta"'] \& \& W \ar[dr,"\delta"] \& \\
\& Y' \ar[from=uu,crossing over,"f'"',near start]\ar[rr,"j'"'] \& \& W',\end{tikzcd}\]
that is an extension in $\Ext_\E^2(\C)$ and if the back face is a double central extension, then the front face is a double central extension as well. Note that being an extension in $\Ext_\E^2(\C)$ - a \emph{triple extension} in $\C$ - means that every face of the cube above is a double extension in $\C$ and, moreover, the induced square
\[\begin{tikzcd}[ampersand replacement=\&] X \ar[r,"g"]\ar[d,"{\langle f,\alpha\rangle}"']\& Z \ar[d,"{\langle h,\gamma \rangle}"]\\ Y\times_{Y'}X' \ar[r]\& W\times_{W'}Z'\end{tikzcd}\]
is a double extension.

All these results apply when $\E$ is the class of regular epimorphisms in $2$-$\Eq(\C)$; indeed, in that case monadic extensions are exactly the same as effective descent morphisms and, moreover, $\Conn(\C)$ is a strongly $\E$-Birkhoff subcategory of $2$-$\Eq(\C)$ by \Cref{prop:Birkhoff}. We will need, however, to restrict our attention to a smaller class of extensions, which we will call \emph{fibrations}; this term was already used in \cite{J04} to denote the classes of arrows occurring in a Galois structure.

\begin{defn}\label{fibration}
A \emph{fibration} $f:(X,R,S) \to (X',R',S')$ is a regular epimorphism  in the category $2$-$\Eq(\C)$ satisfying one of the following equivalent properties :
\begin{itemize}
\item $f^{-1}(R')=R$ and $f^{-1}(S')=S$;
\item $Eq[f]\leq R\wedge S$;
\item the induced arrows $\frac{X}{R}\to \frac{X'}{R'}$ and $\frac{X}{S}\to \frac{X'}{S'}$ are isomorphisms.
\end{itemize}
\end{defn}
From now on $\F$ will denote the class of fibrations in $2$-$\Eq(\C)$. One can check that the class $\F$ satisfies (E1), (E2) and (E3); however, $\F$ does not satisfy (E4), because not every split epimorphism is in $\F$. But the following variant of (E4) holds:

\begin{lemma}Consider the following commutative diagram in $2$-$\Eq(\C)$:
\[\begin{tikzcd}[ampersand replacement=\&] (X,R,S) \ar[rr,"f"]\ar[dr,"h"']\& \& (X',R',S') \ar[dl,"g"] \\ \& (X'',R'',S'').\& 
\end{tikzcd}\]
If $h$ is in $\F$ and $f:X\to X'$ is a regular epimorphism in $\C$, then both $f$ and $g$ are in $\F$.
\end{lemma}

\begin{proof}Taking coequalizers of $R$, $R'$ and $R''$ yields a commutative triangle
\[\begin{tikzcd}[ampersand replacement=\&] \frac{X}{R} \ar[rr,"\overline{f}"]\ar[dr,"\overline{h}"']\& \& \frac{X'}{R'} \ar[dl,"\overline{g}"] \\ \& \frac{X''}{R''}, \& 
\end{tikzcd}\]
in $\C$. Now by assumption $\overline{h}$ is an isomorphism, and thus $\overline{f}$ is a monomorphism. Since by construction it is a regular epimorphism, it is an isomorphism, and as a consequence so is $\overline{g}$. The same argument applies to the quotients by $S,S',S''$, hence $f$ and $g$ are in $\F$.
\end{proof}

As a consequence, we have :
\begin{lemma} Every fibration $f$ induces a monadic pullback functor
\[f^*:2\text{-}\Eq(\C) \downarrow_\F (X',R',S') \to 2\text{-}\Eq(\C) \downarrow_\F (X,R,S).\]
\end{lemma}

\begin{proof}We know already that the pullback functor
\[f^*:2\text{-}\Eq(\C) \downarrow_\E (X',R',S') \to 2\text{-}\Eq(\C) \downarrow_\E (X,R,S)\]
has a left adjoint $f_!$ defined by composition with $f$, and that it is monadic. The class $\F$ is stable under pullbacks and composition: accordingly, the functor $f^*$ and its adjoint $f_!$ restrict to an adjunction between $2$-$\Eq(\C) \downarrow_\F(X',R',S')$ and $2$-$\Eq(\C) \downarrow_\F (X,R,S)$. We need to prove that this restriction is still monadic: using Beck's monadicity theorem \cite{ML}, we only need to prove that $2$-$\Eq(\C) \downarrow_\F (X,R,S)$ is closed in $2$-$\Eq(\C) \downarrow_\E (X,R,S)$ under coequalizers. Since the forgetful functor $2$-$\Eq(\C) \downarrow (X,R,S)\to 2$-$\Eq(\C)$ creates coequalizers, this follows from the previous lemma.
\end{proof}

Observe that each $(X,R,S)$-component of the unit of the reflection $2$-$\Eq(\C)\to \Conn(\C)$ actually belongs to $\F$, and we have shown in the proof of \Cref{prop:Birkhoff} that so is the induced map $$X\to X'\times_{X'/[R',S']} X/[R,S]$$ whenever $f:(X,R,S)\to (X',R',S')$ is a regular epimorphism. Thus $\Conn(\C)$ is also a strongly $\F$-Birkhoff subcategory of $2$-$\Eq(\C)$, and the reflector $I$ preserves arrows in $\F$. Proposition $2$ in \cite{E14} then implies that the Galois structure $\Gamma_\F$ is admissible.

Moreover, an extension is $\Gamma_\F$-trivial if and only if it is $\Gamma_\E$-trivial and lies in $\F$; indeed, both conditions are equivalent to the square \eqref{eq:units_square} being a pullback. Since in both classes any extension is monadic, the same holds for central and normal extensions. Moreover, the centralisation of an extension in $\F$ is again in $\F$, thus the same reasoning works for double central extensions. It follows that double central extensions for $\F$ are exactly the double central extensions for the class of regular epimorphisms which also belong to $\F^1$.

\section{The characterization theorem for central extensions}

As recalled in the Introduction, when the category $\C$ is a Mal'tsev variety of universal algebras, T. Everaert and the second author proved in \cite{EG06} that a surjective homomorphism \eqref{extensiongraph}
 in the category $\RG(\C)/B$ of reflexive graphs in a Mal'tsev variety $\C$ (with fixed ``algebra of objects'' $B$) is central with respect to the reflection to the subcategory $\Grpd(\C)/B$ of groupoids in $\C$ if and only if the \emph{commutator condition}
\begin{equation}\label{cc} [Eq[f],Eq[c]\vee Eq[d]]=\Delta_{X_1} \end{equation}
holds in $\C$. 
We now prove a more general result in any exact Mal'tsev category (with coequalizers) by characterizing the central extensions corresponding to the reflection $2$-$\Eq(\C) \rightarrow  \Conn(\C)$. For this, we first prove the result for the Galois structure $\Gamma_\E=(2\text{-}\Eq(\C),\Conn(\C),I,\E)$, where $\E$ is the class of regular epimorphisms in $2\text{-}\Eq(\C)$.

\begin{lemma}\label{non-trivial}
Let $\C$ be an exact Mal'tsev category with coequalizers, and $f:(X,R,S)\to (X',R',S')$ a regular epimorphism in $2$-$\Eq(\C)$
\[\begin{tikzcd}[ampersand replacement=\&]R \ar[r,shift left]\ar[r,shift right] \ar[d,"f_R"'] \& X \ar[d,"f"] \& \ar[l,shift left]\ar[l,shift right] S \ar[d,"f_S"]\\
R' \ar[r,shift left]\ar[r,shift right] \& X' \& \ar[l,shift left]\ar[l,shift right] S' .\end{tikzcd}\]
If $[Eq[f],R\vee S]=\Delta_X $, then $f$ is a $\Gamma_\E$-normal extension.
\end{lemma}

\begin{proof}
From the definitions above, and because kernel pairs in $2$-$\Eq(\C)$ are computed levelwise and are maximal double equivalence relations, we need to show that in the diagram
\[\begin{tikzcd}[ampersand replacement=\&] R\square Eq[f] \ar[d,shift right]\ar[d, shift left] \ar[r,shift left,"\phi_1"]\ar[r,shift right,"\phi_2"'] \& R \ar[r,"f_R"] \ar[d,shift right]\ar[d, shift left] \& R' \ar[d,shift right]\ar[d, shift left] \\
Eq[f] \ar[r,shift left,"\pi_1"]\ar[r,shift right,"\pi_2"'] \& X \ar[r,"f"]  \& X' \\
S\square Eq[f] \ar[u,shift right]\ar[u, shift left]\ar[r,shift left,"\phi_1'"]\ar[r,shift right,"\phi_2'"'] \& S \ar[u,shift right]\ar[u, shift left] \ar[r,"f_S"'] \& S' \ar[u,shift right]\ar[u, shift left] \end{tikzcd}\]
the extension
$$\pi_1 : (Eq[f],R\square Eq[f],S\square Eq[f]) \to (X,R,S)$$
(or, equivalently, $\pi_2$) is a $\Gamma_\E$-trivial extension in $2$-$\Eq(\C)$ (with respect to the reflection to $\Conn(\C)$).

To this end, we first consider the pushout
\begin{equation}\begin{tikzcd}[ampersand replacement=\&] X \ar[d,"q_R"'] \ar[r,"q_S"]\& \frac{X}{S} \ar[d,"q'"]\\ \frac{X}{R} \ar[r,"q"']\& \frac{X}{R \vee S},\end{tikzcd}\label{eq:pushout}\end{equation}
so that $q'q_S = q_{R\vee S} = qq_R$ (this pushout exists since $\C$ is an exact Mal'tsev category). Then we have
\[Eq[q'q_S]=R\vee S=Eq[qq_R]\]
and thus
\[q_S(R)=Eq[q']\qquad q_R(S)=Eq[q].\]

As a consequence, if the condition $[Eq[f],S]=\Delta_{X}$ holds, then taking the direct image of both sides of this equality by $q_R$ yields
\[ [q_R(Eq[f]),Eq[q]] = \Delta_{\frac{X}{R}}.\]
A similar argument shows that
\[[q_S (Eq[f]),Eq[q']]=\Delta_{\frac{X}{S}}.\]
This observation allows us to consider the diagram

\begin{equation}\begin{tikzcd}[ampersand replacement=\&]
C \ar[dr,shift left]\ar[dr,shift right] \ar[dd] \ar[rr, shift right]\ar[rr,shift left]\& \& R \vee S \ar[dd] \ar[dr,shift left] \ar[dr, shift right] \& \& \& \\
\& {Eq[f]} \ar[rr,"\pi_1", near start,shift left,crossing over]\ar[rr,"\pi_2"',near start, shift right,crossing over] \& \& X\ar[dd,"q_R"] \ar[dr,"q_{R\vee S}"] \ar[rr,"f", near end] \& \& X' \ar[dd,"q_{R'}"] \\
C' \ar[dr,shift left]\ar[dr,shift right]\ar[rr,shift right]\ar[rr, shift left]\& \& {q_R(S)} \ar[dr, shift left]\ar[dr, shift right]\& \& \frac{X}{R\vee S} \&
\\ \& {q_R(Eq[f])} \ar[from=uu,crossing over]\ar[rr,"\psi_1'",shift left]\ar[rr,"\psi_2'"', shift right]\& \& \frac{X}{R} \ar[dr,"q"]\ar[rr,"f_0'"', near end]\& \& \frac{X'}{R'}
\\ \& \& \& \& \frac{X}{R\vee S} \ar[ equal,crossing over,from=uu] \&
\end{tikzcd}\end{equation}
where $C$ and $C'$ are centralizing relations (so that the upper and lower commutative horizontal squares are pullbacks), and the arrow between them is the canonical arrow making the diagram commute. By taking the coequalizers $p$, $p'$ of the equivalence relations $\begin{tikzcd}C \ar[r,shift left] \ar[r,shift right] & Eq[f] \end{tikzcd}$ and $\begin{tikzcd}C' \ar[r,shift left] \ar[r,shift right] & q_R(Eq[f]) \end{tikzcd}$ respectively, we obtain the diagram

$$ \begin{tikzcd}[ampersand replacement=\&,column sep=large]
R\square Eq[f] \ar[dd,shift right] \ar[dd,shift left] \ar[rr,"\phi_1",shift left,near end]\ar[rr,"\phi_2"', shift right,near end] \ar[dr]\& \& R \ar[dd,shift right] \ar[dd,shift left]\ar[rr,"f_R"', near end] \ar[dr]\& \& R' \ar[dd,shift right] \ar[dd,shift left] \\
\& Eq[\rho] \ar[rr, shift left,crossing over, "\widetilde{\phi}_1" near start]\ar[rr,shift right,"\widetilde{\phi}_2"' near start,crossing over] \& \& \frac{X}{R\vee S} \& \\
Eq[f] \ar[dr,"p"'] \ar[rr,"\pi_1",shift left,near start]\ar[rr,"\pi_2"' near start, shift right]\ar[dd]\& \& X\ar[dr,"q_{R\vee S}"]\ar[dd,"q_R", near start] \ar[rr,"f"] \& \& X' \ar[dd,"q_{R'}"] \\
\& \frac{Eq[f]}{C} \ar[from=uu,crossing over,shift right] \ar[from=uu,crossing over,shift left] \ar[rr,shift left,"\widetilde{\pi}_1",near start,crossing over]\ar[rr, shift right,"\widetilde{\pi}_2"',near start,crossing over]\& \& \frac{X}{R\vee S} \ar[from=uu,crossing over,shift right] \ar[from=uu,crossing over,shift left]\& \\
q_R(Eq[f]) \ar[rr,"\psi_1",shift left,near end]\ar[rr,"\psi_2"', shift right,near end] \ar[dr,"p'"']\& \& \frac{X}{R} \ar[rr] \ar[dr,"q"]\& \& \frac{X'}{R'} \\ \& \frac{q_R(Eq[f])}{C'} \ar[from=uu,crossing over,"\rho"' near start] \ar[rr, shift left, "\widetilde{\psi}_1"]\ar[rr,shift right,"\widetilde{\psi}_2"'] \& \& \frac{X}{R\vee S}\ar[from=uu, equal,crossing over] \&  
\end{tikzcd}\label{eq:coegalisateurs}$$
where all the arrows in the front face of the lower cube are induced by the universal property of the coequalizers, so that the whole diagram commutes.

Now, by the so-called Barr-Kock theorem (see \cite{BG04}, for instance), the upper and the lower faces of the lower cube are pullbacks. If we then take the kernel pairs of all the vertical maps, we obtain a pullback of equivalence relations. Proceeding similarly with the roles of $R$ and $S$ reversed, we obtain a pullback
\[\begin{tikzcd}[ampersand replacement=\&] {(Eq[f],R\square Eq[f],S\square Eq[f])} \ar[r,"\pi_1"] \ar[d,"p"']\& {(X,R,S)} \ar[d,"q_{R\vee S}"] \\ {\left(\frac{Eq[f]}{C},Eq[\rho],Eq[\sigma]\right)} \ar[r,"\widetilde{\pi}_1"'] \& {\left(\frac{X}{R\vee S},\Delta,\Delta\right)} 
\end{tikzcd}\]
in $2$-$\Eq(\C)$; since trivial extensions are pullback stable when the subcategory is admissible (\cite{JK97}), to complete the proof all we need to show is that $\widetilde{\pi}_1$ is a $\Gamma_\E$-trivial extension.

This will follow immediately if we show that this extension actually lies in $\Conn(\C)$. Since the smallest equivalence relation $\Delta$ always centralizes itself, it suffices to prove that $Eq[\rho]$ and $Eq[\sigma]$ centralize each other. For this, observe that $\begin{tikzcd}[ampersand replacement=\&]\frac{Eq[f]}{C}\ar[r,shift left,"\tilde{\pi}_1"]\ar[r,shift right,"\tilde{\pi}_2"'] \& \frac{X}{R \vee S} \end{tikzcd}$ is an internal groupoid in $\C$,
as a regular quotient of the equivalence relation $\begin{tikzcd}[ampersand replacement=\&]Eq[f] \ar[r,shift left,"{\pi}_1"]\ar[r,shift right,"{\pi}_2"'] \& X \end{tikzcd}$, so that $[Eq[\widetilde{\pi}_1],Eq[\widetilde{\pi}_2]]=\Delta_{\frac{Eq[f]}{C}}$. Since $\widetilde{\pi}_1=\widetilde{\psi}_1\circ \rho $, we have $Eq[\rho]\leq Eq[\widetilde{\pi}_1]$. A similar argument shows that $Eq[\sigma] \leq Eq[\widetilde{\pi}_2]$, where $\sigma \colon \frac{Eq[f]}{C} \rightarrow \frac{q_S(Eq[f])}{C''}$ and $C''$ is the object part of the centralizing equivalence relation on $q_S(Eq[f])$ and $q_S(R)$. It follows that
\[[Eq[\rho], Eq[\sigma]]\leq [Eq[\widetilde{\pi}_1],Eq[\widetilde{\pi}_2]]=\Delta_{\frac{Eq[f]}{C} },\]
which completes the proof.
\end{proof}

Now, for the Galois structure $\Gamma_\F$, we find :

\begin{thm}\label{Main2} Let $\mathcal C$ be an exact Mal'tsev category with coequalizers, and $f:(X,R,S)\to (X',R',S')$ a fibration in $2$-$\Eq(\C)$.
Then the following conditions are equivalent :
\begin{thmlist}
\item $f$ is a $\Gamma_\F$-central extension;
\item $f$ is a $\Gamma_\F$-normal extension;
\item the commutator condition \eqref{cc} holds: $$[Eq[f],R\vee S]=\Delta_{X}.$$
\end{thmlist}
\end{thm}

\begin{proof}
Any extension which is $\Gamma_\F$-central is also $\Gamma_\E$-central. It is then also $\Gamma_\E$-normal (by Lemma $6$ in \cite{E14}), and thus $\Gamma_\F$-normal. This proves $(1)\Rightarrow (2)$; the converse holds by definition.

For $(2) \Rightarrow (3)$, we recall that if $f$ is normal, then the first projection $\pi_1$ of its kernel pair is trivial. The equivalence relation induced on $Eq[f]$ by $R$ and $R'$ is equal to $Eq[f]\square R$, or equivalently $\pi_1^{-1}(R)\wedge \pi_2^{-1}(R)$; and since $R=f^{-1}(R')$ because $f\in\F$, we have
\[\pi_1^{-1}(R)= \pi_1^{-1}(f^{-1}(R'))=(f\pi_1)^{-1}(R')=(f\pi_2)^{-1}(R')= \pi_2^{-1}(R).\]
In a similar way, one can show that the equivalence relation on $Eq[f]$ induced by $S$ and $S'$ is $\pi_1^{-1}(S)=\pi_2^{-1}(S)$, and thus $\pi_1$ being trivial means that the square
\[\begin{tikzcd}[ampersand replacement=\&,column sep=huge] Eq[f]\ar[d,"\pi_1"']\ar[r,"q_{[\pi_2^{-1}(R),\pi_2^{-1}(S)]}"] \& \frac{Eq[f]}{[\pi_2^{-1}(R),\pi_2^{-1}(S)]} \ar[d,"\overline{\pi_1}"] \\ X \ar[r,"q_{[R,S]}"'] \& \frac{X}{[R,S]} \end{tikzcd}\]
is a pullback; in particular, this implies that
\[Eq[\pi_1]\wedge [\pi_2^{-1}(R),\pi_2^{-1}(S)]=\Delta_{Eq[f]}.\]
By the properties of the Pedicchio commutator, we then have
\begin{eqnarray*} 
[Eq[\pi_1],\pi_2^{-1}(R)] & \leq & [\pi_1^{-1}(S),\pi_2^{-1}(R)]  \\
& = &[\pi_2^{-1}(S),\pi_2^{-1}(R)] \\
& = &[\pi_2^{-1}(R),\pi_2^{-1}(S)] 
\end{eqnarray*}
and
\[[Eq[\pi_1],\pi_2^{-1}(R)]\leq Eq[\pi_1],\]
hence $[Eq[\pi_1],\pi_2^{-1}(R)]=\Delta_{Eq[f]}$; taking the direct images by $\pi_2$ yields $$[Eq[f],R]=\Delta_{X}.$$ A similar argument shows that $[Eq[f],S]=\Delta_{X}$, and then
\[[Eq[f],R\vee S]=[Eq[f],R]\vee [Eq[f],S]=\Delta_X.\]
The result then follows from \cref{non-trivial}, and the fact that a fibration is $\Gamma_\E$-normal if and only if it is $\Gamma_\F$-normal.
\end{proof}

\begin{remark} For a regular epimorphism which does not lie in $\F$, it is not necessarily true that any $\Gamma_\E$-central extension must satisfy the commutator condition \eqref{cc} in Theorem \ref{Main2}. For example, for any object $X$ in $\C$, we have that $(X,\Delta_X,\nabla_X)$ is in $\Conn(\C)$, since
\[[\Delta_X,\nabla_X]\leq \Delta_X\wedge \nabla_X=\Delta_X.\]
Thus any regular epimorphism $f:X\to X'$ gives a $\Gamma_\E$-trivial (hence $\Gamma_\E$-central) extension in $2$-$\Eq(\C)$, but in general the condition $[Eq[f],\nabla_X]=\Delta_X$ is not satisfied.
\end{remark}

\begin{example}[Abelian objects]\label{ex:abelian}The characterization of central extensions with respect to the category of abelian objects, which was proved by G. Janelidze and G.M. Kelly for the case of Mal'tsev varieties \cite{JK00} can be seen as a special case of \Cref{Main2}.

Indeed, for every object $X$ of $\C$, $(X,\nabla_X,\nabla_X)$ is an object of $2$-$\Eq(\C)$; via this identification we can see $\C$ as a subcategory of $2$-$\Eq(\C)$, which is full since $f(\nabla_X)\leq \nabla_{X'}$ for any arrow $f:X\to X'$ in $\C$. It is also easily seen that it is closed under limits and regular quotients. Moreover $(X, \nabla_X,\nabla_X)$ is in $\Conn(\C)$ if and only if $[\nabla_X,\nabla_X]=\Delta_X$ if and only if $X$ is an abelian object of $\C$; and the quotient $X\to \frac{X}{[X,X]}$ serves as $X$-reflection for both cases. Since the condition $Eq[f]\leq \nabla_X$ is trivially satisfied for any regular epimorphism $f$ of $\C$, an extension $f:X\to X'$ is central in $\C$ with respect to $\Ab(\C)$ if and only if $f:(X,\nabla_X,\nabla_X)\to (X',\nabla_{X'},\nabla_{X'})$ is a $\Gamma_\F$-central extension of $2$-$\Eq(\C)$ with respect to $\Conn(\C)$. Thus in this case, \Cref{Main2} shows that an extension is central if and only if $[Eq[f],\nabla_X]=\Delta_X$.
\end{example}

\begin{example}[Spans and pregroupoids]\label{ex:spans}
Since the category $\C$ is assumed to be exact, the data of two equivalence relations $R,S$ on $X$ is equivalent to the data of two regular epimorphisms with domain $X$. In other words our category $2$-$\Eq(\C)$ is equivalent to the category $\E\text{-}\Span(\C)$ of spans 
\begin{equation}\label{span} \begin{tikzcd}[ampersand replacement=\&] X_0 \& \ar[l,"\alpha"'] X \ar[r,"\beta"]\&  X_0'\end{tikzcd} \end{equation}
whose morphisms $\alpha$ and $\beta$ are regular epimorphisms in $\C$. Through this equivalence the subcategory $\Conn(\C)$ corresponds to the subcategory whose objects are spans \eqref{span}
 where $[Eq[\alpha],Eq[\beta]]=\Delta_X$ or, equivalently, where there exists a partial Mal'tsev operation $p:X\times_{X_0} X\times_{X_0'} X\to X$.

Such a structure on a span can be defined in any category with finite limits. When it also satisfies the associativity axiom $p(p(x,y,z),u,v)=p(x,y,p(z,u,v))$ (which, as we mentioned before, automatically holds in a Mal'tsev category), it forms what has been called a \emph{pregroupoid} by Kock \cite{K87,K89} and later a \emph{herdoid} by Johnstone \cite{Jo91}. As noted in Definition \ref{fibration}, the condition that an arrow $f:\X\to \Y$ belongs to $\F$ is equivalent to the fact that the induced arrows between the codomains are isomorphisms, so morphisms in $\F$, seen as morphisms in $\E\text{-}\Span(\C)$, must be of the form
\[\begin{tikzcd}[ampersand replacement=\&]\& \ar[dl,"\alpha"'] X \ar[dr,"\beta"]\ar[dd,"f"]\& \\  X_0 \& \&  X_0' \\ \& \ar[ul,"\alpha'"] Y \ar[ur,"\beta'"'] \& \end{tikzcd}\]
with $f$ a regular epi in $\C$. Then our Theorem \ref{Main2} says that an extension of this form in $\E\text{-}\Span(\C)$ is central with respect to the subcategory of pregroupoids if and only if $[Eq[f],Eq[\alpha]\vee Eq[\beta]]=\Delta_X$.

As a special case, one can then consider spans where the two structural morphisms are split epimorphisms, with a given common section; thus we have $\X=(X_1,X_0,c,d,i)$ where $di = id_{X_0}= ci$ in $\C$. These are the reflexive graphs in $\C$, and a pregroupoid structure on a reflexive graph is in fact equivalent to an internal groupoid structure.

We can then consider the category $\RG(\C)/X_0$ of reflexive graphs over a fixed ``object of objects'' $X_0$, where \emph{all} arrows are of the form
\[\begin{tikzcd}[ampersand replacement=\&] X_1 \ar[rr,"f"] \ar[dr,shift left,"d"] \ar[dr,shift right,"c"'] \& \& Y_1 \ar[dl,shift left,"d'"] \ar[dl,shift right,"c'"'] \\  \& X_0 \&\end{tikzcd}\]
and also satisfy $fi=i'$. This category can be seen as a (non full) subcategory of $\E\text{-}\Span(\C)$, and it can be shown that it is closed under pullbacks. Moreover, the condition that $X_0$ is fixed implies that $Eq[f]\leq  Eq(c) \wedge Eq(d)$, so that regular epimorphisms in $\RG(\C)/X_0$ all lie in the class $\F$ of fibrations in $2$-$\Eq(\C)$. We conclude that an extension is central in $\RG(\C)/X_0$ if and only if it is $\Gamma_\F$-central when seen in $2$-$\Eq(\C)$. Thus, by \Cref{Main2}, an extension is central in $\RG(\C)/X_0$ if and only if
\[[Eq[f],Eq[c]\vee Eq[d]]=\Delta_{X_1}.\]
As mentioned above, this characterization was proved by T. Everaert and the second author in \cite{EG06} in the special case of Mal'tsev varieties, although the proof involved long calculations with elements, which could not be carried in abstract categories.
\end{example}

\begin{remark}
When the category $\C$ is pointed, so that there is an object $0$ in $\C$ that is both initial and terminal, then every object $X$ has a unique structure of reflexive graph above $0$, so that $\RG(\C)/0\simeq \C$. Since the kernel pair of the unique arrow $X\to 0$ is the largest equivalence relation $\nabla_X$ on $X$, such a reflexive graph is an internal groupoid if and only if $X$ is an abelian object, so that $\Grpd(\C)\simeq \Ab(\C)$; and in this context this is equivalent to the fact that $X$ is an abelian group object. In such a category, we can then obtain the characterisation of central extensions in $\C$ with respect to $\Ab(\C)$ as the special case of the characterisation of central extensions in $\RG(\C)/X_0$ with respect to $\Grpd(\C)/X_0$, where $X_0=0$.

This isn't true, however, when the category is not pointed, since in this case defining an internal groupoid above the terminal object $1$ requires the choice of an ``identity'' $1\to X$, which may not always be possible (for example, in the category of unital rings, the only ring $X$ for which there is an arrow $1\to X$ is the terminal ring itself). In such a case it is thus necessary to consider spans and pregroupoids or, equivalently, pairs of equivalence relations and connectors, to obtain a context that encompasses both the cases of internal groupoids and of abelian objects. This was our main motivation to consider the Galois structure on $2$-$\Eq(\C)$ rather than $\RG(\C)$.
\end{remark}

\begin{example}\label{compact}
The category $\Grp(\Comp)$ of compact (Hausdorff) groups is an exact Mal'tsev category which is monadic over the category of sets \cite{M76}. As it follows from Lemma $6.13$ in \cite{EG15}, the Huq commutator of two normal subobjects (=closed normal subgroups) $H$ and $K$ of a compact group $G$ is simply given by the closure $\overline{[H,K]}$ of the classical group-theoretic commutator $[H,K]$ of $H$ and $K.$ In the category $\Grp(\Comp)$ this Huq commutator of normal subobjects is the normal subobject associated with the categorical commutator \cite{P95} of the corresponding equivalence relations: this essentially follows from \cite{BG02}, by taking into account the fact that $\Grp(\Comp)$ is a strongly protomodular category \cite{B04}. From the previous example, it then follows that an extension
\[\begin{tikzcd}[ampersand replacement=\&] X_1 \ar[rr,"f"] \ar[dr,shift left,"d"] \ar[dr,shift right,"c"']\& \& Y_1 \ar[dl,shift left,"d'"] \ar[dl,shift right,"c'"'] \\ \& X_0 \& \end{tikzcd}\]
in $\RG( \Grp(\Comp) )/X_0$ is \emph{central} relatively to $\Grpd( \Grp(\Comp) )/X_0$ if and only if $\overline{[Ker [f] , Ker [d] \cdot Ker[c]]}=\{1\}$, where $Ker [f]$, $Ker [d]$ and $Ker [c]$ denote the kernels of $f$, $d$ and $c$ in the category $\Grp(\Comp)$, respectively, while $Ker [d] \cdot Ker[c]$ is the group-theoretic product of these normal closed subgroups, which is a compact (and thus closed) subgroup, since it is the image of the multiplication map
\[\cdot : Ker[d]\times Ker[c]\to X_1.\]
\end{example}

\section{Double central extensions}

As we explained in \cref{preliminaries}, the category $\CExt_\F(2$-$\Eq(\C))$ of $\Gamma_\F$-central extensions is reflective in $\Ext_\F(2$-$\Eq(\C))$, and \emph{double central extensions} are double extensions that are central with respect to this induced reflection.
We shall show that the double extensions that are central with respect to this induced reflection can be characterized by some natural conditions involving Smith-Pedicchio commutators.

These conditions are a generalization of the ones given by G. Janelidze in \cite{J91} characterizing double central extensions in $\Grp$, and later extended to Mal'tsev varieties by the second author and V. Rossi in \cite{GR04}, and to exact Mal'tsev categories by T. Everaert and T. Van der Linden \cite{EV10}. In fact, the proofs given below are suitably adapted from the ones appearing in these two papers.

\begin{lemma}\label{surjectivity} Consider the following pullback of a double extension $(g,j)$ along a double extension $(\gamma,\delta)$ in $2$-$\Eq(\C)$ depicted as
\begin{equation} \begin{tikzcd}[ampersand replacement=\&]
\X\times_{\Z} \U\ar[dr,"\alpha"']\ar[dd,"f'"']\ar[rr,"g'"] \& \& \U \ar[dr,"\gamma"]\ar[dd,"h'", near end] \& \\
\& \X \ar[rr,"g",crossing over,near start]\& \& \Z\ar[dd,"h"]\\
\Y\times_\W \V \ar[rr,"j'",near end] \ar[dr,"\beta"'] \& \& \V \ar[dr,"\delta"] \& \\
\& \Y \ar[from=uu,crossing over,"f"',near start]\ar[rr,"j"'] \& \& \W\end{tikzcd}\label{eq:PB_double}
\end{equation}
Then we have
$$\alpha([Eq[f'],Eq[g']])=[Eq[f],Eq[g]]$$
and
$$\alpha([Eq[f']\wedge Eq[g'],R'\vee S'])=[Eq[f]\wedge Eq[g],R\vee S],$$
where $\X=(X,R,S)$ and $\X\times_{\Z} \U=(X\times_{Z} U,R',S')$.
\end{lemma}

\begin{proof} Since morphisms in $\F$ are regular epimorphisms, and the pullbacks and regular epimorphisms in $2$-$\Eq(\C)$ are computed ``levelwise'', the cube \eqref{eq:PB_double} induces a cube in $\C$ that is also a pullback of regular pushouts.

Since the Pedicchio commutator is preserved by regular images in $\C$, all we have to show is that
$$\alpha(Eq[f'])=Eq[f], \quad \alpha(Eq[g'])=Eq[g], \quad \alpha(R')=R,$$
$$\alpha(S')=S\quad \text{and} \quad \alpha(Eq[f']\wedge Eq[g'])=Eq[f]\wedge Eq[g].$$
This holds by assumption for $R$ and $S$; by the property \eqref{property} recalled in \cref{preliminaries}, it suffices to prove in each case that the corresponding coequalizers are part of a regular pushout in $\C$. This is true for $f$ because the square
\[\begin{tikzcd}[ampersand replacement=\&] X\times_{Z} U\ar[r,"\alpha"]\ar[d,"f'"'] \& X\ar[d,"f"] \\ Y\times_{W} U \ar[r,"\beta"']\& Y \end{tikzcd}\]
is the pullback of a regular pushout in $\C$, and is thus a regular pushout itself, since regular pushouts are pullback stable in $\C$. It is also true for $g$ because the corresponding square (the top face of \eqref{eq:PB_double}) is a pullback.

Thus we only have to treat the case of the intersection $Eq[f]\wedge Eq[g]$; its coequalizer is the arrow $X\to Y\times_{W} Z$ induced by $f$ and $g$. Accordingly, the corresponding square of coequalizers is the left square in the commutative diagram

\[\begin{tikzcd}[ampersand replacement=\&] X \times_{Z}U \ar[r,"{\langle f',g'\rangle}"] \ar[d,"\alpha"] \& (Y\times_{W} V)\times_{V}U \ar[d,"\bar{\alpha}"]\ar[r]\& U \ar[d,"\gamma"]\\
X \ar[r,"{\langle f,g\rangle}"'] \& Y\times_{W} Z \ar[r] \& Z.\end{tikzcd}\]

Now the whole rectangle above is the top face of \eqref{eq:PB_double}, while the right square is the top face of the cube

\[\begin{tikzcd}[ampersand replacement=\&] (Y\times_{W}V)\times_{V}U \ar[dr,"\bar{\alpha}"']\ar[dd]\ar[rr] \& \& U \ar[dr,"\gamma"]\ar[dd,"h'",near end] \& \\
\& Y\times_{W} Z \ar[rr,crossing over]\& \& Z\ar[dd," h"]\\
Y\times_{W}V \ar[rr," j'",near start] \ar[dr,"\beta"'] \& \& V \ar[dr,"\delta"] \& \\
\& Y\ar[rr," j"'] \ar[from=uu,crossing over]\& \& W,\end{tikzcd}\]
and is then also a pullback, since the front and back faces are pullbacks by construction.\end{proof}

\begin{prop}In the diagram \eqref{eq:PB_double}, the front face satisfies the commutator conditions
\begin{equation}[Eq[f],Eq[g]]=\Delta_{X}=[Eq[f]\wedge Eq[g], R\vee S].\label{eq:double_commutator}\end{equation}
if and only if the back face does, i.e.
\begin{equation}[Eq[f'],Eq[g']]=\Delta_{X'}=[Eq[f']\wedge Eq[g'], R'\vee S'].\end{equation}
\label{prop:commutator_conditions}
\end{prop}

\begin{proof}Suppose first that the back face satisfies the commutator conditions; then by Lemma \ref{surjectivity}
\[[Eq[f], Eq[g]]=\alpha([Eq[f'],Eq[g']])=\Delta_{X},\]
and similarly
\[[Eq[f]\wedge Eq[g],R\vee S]=\alpha([Eq[f']\wedge Eq[g'],R'\vee S'])=\Delta_{X}.\]

Assume then that the front face satisfies the commutator conditions; then since $\alpha([Eq[f'], Eq[g']])=\Delta_{X}$, we have
\[[Eq[f'] , Eq[g']]\leq Eq[\alpha];\]
and, moreover,
\[[Eq[f'],Eq[g']]\leq Eq[g']\]by property $(2)$ of the categorical commutator. Since the top face in the cube \eqref{eq:PB_double} is a pullback, $Eq[g']\wedge Eq[\alpha]=\Delta_{X\times_{Z} U}$, and thus
\[[Eq[f'],Eq[g']]=\Delta_{X\times_{Z} U}.\]
The proof for the other condition is similar.
\end{proof}

\begin{coro}Any double $\Gamma_\F$-central extension satisfies the commutator conditions \eqref{eq:double_commutator} in Proposition \ref{prop:commutator_conditions}.
\label{implication_1}
\end{coro}

\begin{proof}
By definition:
\begin{itemize}\item $(g,j)$ is a double $\Gamma_\F$-central extension if and only if it is split by a double extension $(\gamma,\delta)\in \F^2$;
\item $(g',j')$ is a double $\Gamma_\F$-trivial extension if and only if it is the pullback of some extension between $\Gamma_\F$-central extensions.\end{itemize}

Thus $(g,j)$ is a double $\Gamma_\F$-central extension if and only if there are two pullback squares in $\Ext_\F(2$-$\Eq(\C))$
\begin{equation*}
\begin{tikzcd}[ampersand replacement=\&] f' \ar[r,"{(g',j')}"] \ar[d,"{(\alpha,\beta)}"'] \& h'\ar[d,"{(\gamma,\delta)}"] \\ f \ar[r,"{(g,j)}"']\& h.
\end{tikzcd} \qquad \begin{tikzcd}[ampersand replacement=\&] f' \ar[r,"{(g',j')}"] \ar[d,"{(\alpha',\beta')}"'] \& h'\ar[d,"{(\gamma',\delta')}"] \\ f'' \ar[r,"{(g'',j'')}"']\& h''.
\end{tikzcd}
\end{equation*}
where $f''$ and $h''$ are $\Gamma_\F$-central extensions in $2$-$\Eq(\C)$. By \cref{prop:commutator_conditions}, $(g,j)$ satisfies the commutator conditions if and only if $(g',j')$ does if and only if $(g'',j'')$ does.

But since $f''$ is a central extension, we have by \cref{non-trivial}
\[[Eq[f'']\wedge Eq[g''],R''\vee S'']\leq [Eq[f''],R''\vee S'']\leq \Delta_{X''}\]
and
\[[Eq[f''],Eq[g'']]\leq [Eq[f''],R''\wedge S'']\leq \Delta_{X''},\]
where $Eq[g'']\leq R''\wedge S''$ because $g''$ is in $\F$.\end{proof}

We now turn to the proof of the converse implication.

\begin{prop} Any double extension in $\Ext_\F^2(\C)$ satisfying the commutator conditions \eqref{eq:double_commutator} is a double $\Gamma_\F$-central extension.\label{implication_2}
\end{prop}

\begin{proof}
As a consequence of the "denormalized $3\times 3$ lemma" from \cite{B03}, if we take a double extension \eqref{eq:double_ext} and take all the kernel pairs horizontally and vertically in $2$-$\Eq(\C)$, and then the kernel pairs of the induced arrow, we obtain a diagram
\begin{equation}
\begin{tikzcd}[ampersand replacement=\&] Eq[f]\square Eq[g]\ar[r,shift left,"q_1"]\ar[r,shift right,"q_2"'] \ar[d,shift left,"p_2"]\ar[d,shift right,"p_1"'] \& Eq[g] \ar[r,"\bar{f}"] \ar[d,shift left] \ar[d,shift right]\& Eq[j] \ar[d,shift left] \ar[d,shift right]\\
Eq[f] \ar[r,shift left]\ar[r,shift right] \ar[d,"\bar{g}"']\& \X \ar[r,"f"]\ar[d,"g"] \& \Y \ar[d,"j"] \\
Eq[h] \ar[r,shift left]\ar[r,shift right] \& \Z \ar[r,"h"'] \& \W.
\end{tikzcd}\label{eq:3x3}
\end{equation}

As explained in \cref{preliminaries}, the condition that $[Eq[f],Eq[g]]=\Delta_{X}$ is equivalent to the existence of a section for the comparison arrow $Eq[f]\square Eq[g]\to Eq[f]\times_X Eq[g]$ in $\C$ satisfying certain conditions. Since $\F$ is stable under pullback, all the arrows involved induce isomorphisms on the quotients by $R$ and $S$; thus this section in $\C$ induces a similar section at the level of $R$ and $S$. Thus we have a section in $2$-$\Eq(\C)$, and thus a connector $p:Eq[f]\times_{\X}Eq[g]\to \X$. In particular, we have the following pullback of split epimorphisms in $2$-$\Eq(\C)$:
\[\begin{tikzcd}[ampersand replacement=\&] Eq[f] \square Eq[g] \ar[d,"x"'] \ar[r,"\pi"] \& Eq[f]\times_{\X} Eq[g] \ar[d,"p"] \\
Eq[f]\wedge Eq[g] \ar[r,"p_1"'] \& \X.
\end{tikzcd}\]
Now the condition $[Eq[f]\wedge Eq[g],R\vee S]=\Delta_{X}$ exactly means that the arrow $\X \to \Y\times_{\W} \Z$ is a $\Gamma_\F$-central - and then $\Gamma_\F$-normal - extension (\Cref{Main2}). It follows in particular that the first projection $p_1:Eq[f]\wedge Eq[g]\to \X$ of its kernel pair is a $\Gamma_\F$-trivial extension. As a consequence, the arrow $\pi: Eq[f]\square Eq[g]\to Eq[f]\times_{\X}Eq[g]$ is also a $\Gamma_\F$-trivial extension.

Now consider the following cube:

\[\begin{tikzcd}[ampersand replacement=\&]
Eq[f]\square Eq[g] \ar[dr,"\eta_{Eq[f]\square Eq[g]}"']\ar[dd,"q_1"']\ar[rr,"p_1"] \& \& Eq[f] \ar[dr,"\eta_{Eq[f]}"]\ar[dd,"\pi_1",near end] \& \\
\& I(Eq[f]\square Eq[g]) \ar[rr,"I(p_1)",crossing over,near start]\& \& I(Eq[f])\ar[dd,"I( \pi_1)"]\\
Eq[g] \ar[rr,"\pi_1"',near end] \ar[dr,"\eta_{Eq[g]}"'] \& \& \X \ar[dr,"\eta_\X"] \& \\
\& I(Eq[g]) \ar[from=uu,crossing over,"I(q_1)"',near start]\ar[rr,"I(\pi_1)"'] \& \& I(\X)
\end{tikzcd}\]

Since $\Conn(\C)$ is a strongly Birkhoff subcategory of $2$-$\Eq(\C)$ and every regular epimorphism is an effective descent morphism, the reflector $I:2$-$\Eq(\C)\to \Conn(\C)$ preserves pullbacks of split epimorphisms along split epimorphisms (see \cite{E14}), so that
$$I(Eq[f]\times_\X Eq[g])=I(Eq[f])\times_{I(\X)} I(Eq[g]).$$

Moreover, since $\pi$ is a $\Gamma_\F$-trivial extension, the induced square to the pullbacks
\[\begin{tikzcd}[ampersand replacement=\&] Eq[f]\square Eq[g]\ar[r,"\pi"]\ar[d,"\eta_{Eq[f]\square Eq[g]}"'] \& Eq[f]\times_\X Eq[g] \ar[d,"\eta_{Eq[f]\times_\X Eq[g]}"]\\ I(Eq[f]\square Eq[g])\ar[r,"I(\pi)"'] \& I(Eq[f]\times_\X Eq[g])\end{tikzcd}\]
is itself a pullback, and thus the whole cube above is the limit of the diagram formed by its front, right and bottom faces. In particular, the induced square
\[\begin{tikzcd}[ampersand replacement=\&] Eq[f]\square Eq[g] \ar[r,"p_1"] \ar[d]\& Eq[f] \ar[d] \\ Eq[g]\times_{I(Eq[g])} I(Eq[f]\square Eq[g]) \ar[r]\& \X \times_{I(\X)} I(Eq[f])\end{tikzcd}\]
is also a pullback, and thus the square
\[\begin{tikzcd}[ampersand replacement=\&] Eq[f]\square Eq[g] \ar[r,"p_1"] \ar[d,"q_1"']\& Eq[f] \ar[d,"\pi_1"] \\ Eq[g] \ar[r,"\pi_1"']\& \X\end{tikzcd}\]
is a pullback of a double extension between two $\Gamma_\F$-trivial (hence $\Gamma_\F$-central) extension; it is thus a double $\Gamma_\F$-trivial extension. Since it is the kernel pair of the (horizontal) double extension
\[\begin{tikzcd}[ampersand replacement=\&] Eq[f] \ar[r,"\bar{g}"] \ar[d,"\pi_1"']\& Eq[h] \ar[d,"\pi_1"] \\ \X \ar[r,"g"']\& \Z\end{tikzcd}\]
this latter extension is $\Gamma_\F$-normal, and then $\Gamma_\F$-central. We can then consider the cube
\[\begin{tikzcd}[ampersand replacement=\&]
Eq[f]\ar[dr,"\pi_2"']\ar[dd,"\pi_1"']\ar[rr,"\bar{g}"] \& \& Eq[h] \ar[dr,"\pi_2"]\ar[dd,"\pi_1",near end] \& \\
\& \X \ar[rr,"g",crossing over,near start]\& \& \Z\ar[dd,"h"]\\
\X \ar[rr,"g"',near end] \ar[dr,"f"'] \& \& \Z \ar[dr,"h"] \& \\
\& \Y \ar[from=uu,crossing over,"f"',near start]\ar[rr,"j"'] \& \& \W.
\end{tikzcd}\]
Since it is a triple extension (as a pullback of double extensions) and the back face is a double $\Gamma_\F$-central extension, the front face is a double $\Gamma_\F$-central extension, which concludes the proof.
\end{proof}

We mentioned above that the commutator conditions were a generalisation of those given for the double central extensions relative to the subcategory of abelian objects \cite{EV10}. In fact, that characterisation is a special case of ours:
\begin{coro}\label{double-abelian} Let $\C$ be an exact Mal'tsev category with coequalisers. Then a double extension is central with respect to the subcategory of abelian objects in $\C$ if and only if
\[[Eq[f],Eq[g]]=\Delta_X=[Eq[f]\wedge Eq[g],\nabla_X].\]
\end{coro}

\begin{proof}As we have explained for one-dimensional extensions in \Cref{ex:abelian}, the reflection of $\C$ into its subcategory of abelian objects is the special case where $R=\nabla_X=S$. In particular, the results above give us that a double extension is central if and only if
\[[Eq[f]\wedge Eq[g],\nabla_{X}]=[Eq[f]\wedge Eq[g],\nabla_X \vee \nabla_X ] = \Delta_X=[Eq[f],Eq[g]].\]
\end{proof}

\begin{coro}\label{doublecentralgraph}
A double extension
\begin{equation}\label{eq:double_ext_graph}\begin{tikzcd}[ampersand replacement=\&] \X\ar[d,"f"'] \ar[r,"g"]\& \Z\ar[d,"h"] \\ \Y\ar[r,"j"'] \& \W\end{tikzcd}\end{equation}
in $\RG(\C)/X_0$ is central if and only if
\[[Eq[f]\wedge Eq[g],Eq[c]\vee Eq[d]] = \Delta_{X_1}=[Eq[f],Eq[g]],\]
where $\X=(X_1,X_0,c,d,i)$.
\end{coro}

\begin{proof}
As in \Cref{ex:spans}, any double extension in $\RG(\C)/X_0$ induces a double extension in $2$-$\Eq(\C)$, and it is $\Gamma_\F$-central if and only if the double extension \eqref{eq:double_ext_graph} is central with respect to the reflection $\RG(\C)/X_0\to \Grpd(\C)/X_0$. Thus it suffices to take $R=Eq[c]$ and $S=Eq[d]$.
\end{proof}

\begin{example} Let us consider once again the category $\Comp(\Grp)$ of compact (Hausdorff) groups. The meet $Ker(f)\wedge Ker(g)$ of two kernels is their set-theoretic intersection, since the intersection of two closed subgroups is a closed (hence compact) subgroup. We then find that a double extension \eqref{eq:double_ext_graph} is central if and only if
\[\overline{[Ker(f)\wedge Ker(g),Ker(c)\cdot Ker(d)]}=\{1\}=\overline{[Ker(f),Ker(g)]}.\]
\end{example}

\begin{example}\label{Lie}
Let $B$ be a Lie algebra; we recall that a precrossed $B$-module of Lie algebras is a triple $(L,\xi,\partial)$ where $L$ is a Lie algebra, $\xi:B\times L\to L:(b,l)\mapsto {^b l}$ is an action of Lie algebras, and $\partial:L\to B$ is a morphism of Lie algebras such that $\partial ( {^b} l)=[b,\partial(l)]$ for all $b\in B$ and $l\in L$. A precrossed $B$-module is said to be a crossed module if, moreover, it satisfies Peiffer's identity $[l,l']= {^{\partial (l)}}l'$ for all $l,l'\in L$. It is known that the category of precrossed $B$-modules of Lie algebras is equivalent to the category $\RG(\Lie)/B$ of reflexive graphs in Lie algebras over $B$, and that this equivalence restricts to an equivalence between crossed modules and groupoids over $B$ \cite{LR80,P87}. Thus our results can be used to characterize central extensions relative to the reflection $\PXMod(\Lie)/B\to \XMod(\Lie)/B$.

Let $f:(L,\xi,\partial)\to (L',\xi',\partial')$ be a surjective morphism of precrossed $B$-modules. Then under the equivalence, $f$ becomes the regular epimorphism
\begin{equation}\begin{tikzcd}[ampersand replacement=\&] B\ltimes L \ar[rr,"1_B\ltimes f"]\ar[dr,shift left,"\pi_B"]\ar[dr, shift right,"{(1,\partial)}"']\& \& B\ltimes L' \ar[dl,shift left, "\pi_B"]\ar[dl, shift right,"{(1,\partial')}"']\\ \& B \& \end{tikzcd} \label{graphes_Lie} \end{equation}
(where $1_B\ltimes f$ is the map $(b,l)\mapsto (b,f(l))$ and $(1_B,\partial)$ is the map $(b,l)\mapsto b+\partial{l}$, for any $b\in B$ and any $l\in L$).

Now, by the results of \Cref{ex:spans}, $1_B\ltimes f$ is a central extension if and only if
$$[Eq[1_B\ltimes f],Eq[\pi_B]\vee Eq[(1,\partial)] ]=\Delta_{B\ltimes L}.$$
In the category of Lie algebras the commutator of equivalence relations is the equivalence relation corresponding to the commutator of the corresponding ideals (\cite{S76}), and this condition is then equivalent to
\begin{equation}\label{eq:comm_croise} [Ker(1_B\ltimes f),Ker(\pi_B)+Ker((1_B,\partial))]=0 .\end{equation}
One verifies that the object part of the kernel of $1_B\ltimes f$ is $Ker(f)$, with inclusion given by the composite $Ker(f)\to L\to B\ltimes L$, and that the kernel of $\pi_B$ is simply $L$. Moreover, one easily sees that
$$Ker((1_B,\partial)) = \{(-\partial(l),l)\mid l\in L\}.$$

As a consequence, we find that $[Ker(1_B\ltimes f),Ker(\pi_B)]$ is the ideal of $B\ltimes L$ generated by terms of the form
\[[(0,k),(0,l)]=(0, [k,l])\]
for $k\in Ker(f)$ and $l\in L$, while $[Ker(1_B\ltimes f),Ker((1_B,\partial))]$ by terms of the form
\[[(0,k),(-\partial(l),l)]=(0, ^{\partial(l)}k +[k,l]).\]
It follows that the commutator described in \eqref{eq:comm_croise} can be seen as an ideal of $L$, more precisely the subspace generated by terms of the form $[k,l]$ or $ ^{\partial(l)} k$, for $k\in Ker(f)$ and $l\in L$. By analogy with the case of precrossed modules of groups \cite{CE89}, we call this subobject the \emph{Peiffer commutator} $\langle Ker(f),L\rangle$ of $Ker(f)$ and $L$.

Let us now consider a double extension
\begin{equation}\begin{tikzcd}[ampersand replacement=\&] (L_1,\partial_1,\xi_1) \ar[r,"g"] \ar[d,"f"'] \& (L_2,\partial_2,\xi_2) \ar[d,"h"]\\ (L_3,\partial_3,\xi_3) \ar[r,"j"'] \& (L_4,\partial_4,\xi_4)
\end{tikzcd}\label{eq:double_ext_Lie}\end{equation}
of precrossed $B$-modules of Lie algebras. Using again the equivalence between $B$-precrossed modules and internal reflexive graphs over $B$ in $\Lie$, we find that \eqref{eq:double_ext_Lie} is central if and only if
\[\begin{tikzcd}[ampersand replacement=\&] B\ltimes L_1 \ar[r,"{1_B\ltimes g}"] \ar[d,"{1_B\ltimes f}"'] \& B\ltimes L_2 \ar[d,"{1_B\ltimes h}"]\\ B\ltimes L_3 \ar[r,"{1_B\ltimes j}"'] \& B\ltimes L_4
\end{tikzcd}\]
is a double central extension in $\RG(\Lie)/B$ (with the graph structures on each object given as in \eqref{graphes_Lie}), which by \Cref{doublecentralgraph} is equivalent to
\begin{equation} [Ker(1_B\ltimes f)\cap Ker(1_B\ltimes g),Ker(\pi_B)+Ker((1_B,\partial_1))]=0\label{eq:comm_croise_2_1}\end{equation}
and
\begin{equation}[Ker(1_B\ltimes f),Ker(1_B\ltimes g)]=0.\label{eq:comm_croise_2_2}\end{equation}
Again, the object parts of the kernels of $1_B\ltimes f$ and $1_B\ltimes g$ are simply $Ker(f)$ and $Ker(g)$ respectively, and so the commutator in equation \eqref{eq:comm_croise_2_2} is generated by terms of the form
$$[(0,k),(0,k')]=(0,[k,k']),$$
where $k\in Ker(f)$ and $k'\in Ker(g)$. Since the commutator in equation \eqref{eq:comm_croise_2_1} can be treated exactly as the commutator in \eqref{eq:comm_croise}, we find that the double extension \eqref{eq:double_ext_Lie} is central if and only if
$$\langle Ker(f)\wedge Ker(g),X \rangle =0=[Ker(f),Ker(g)],$$
where both sides are ideals of $L_1$.
\end{example}

\bibliography{biblio}
\bibliographystyle{abbrv}

\end{document}